\documentclass[a4paper,leqno,12pt, reqno]{amsart}

\usepackage{amsmath}
\usepackage{amssymb}
\usepackage{graphicx}
\usepackage{tikz-cd}
\usepackage{hyperref}
\usepackage[all]{xy}

\usepackage{xspace}
\usepackage{bm}
\usepackage{amsmath}
\usepackage{amstext}
\usepackage{amsfonts}
\usepackage[mathscr]{euscript}
\usepackage{amscd}
\usepackage{latexsym}
\usepackage{amssymb}
\usepackage{enumerate}
\usepackage{xcolor}

\theoremstyle{plain}
\swapnumbers
    \newtheorem{theorem}[figure]{Theorem}
    \newtheorem{proposition}[figure]{Proposition}
    \newtheorem{lemma}[figure]{Lemma}
    
    \newtheorem{corollary}[figure]{Corollary}
    
    \newtheorem{subsec}[figure]{}

\theoremstyle{definition}
    \newtheorem{definition}[figure]{Definition}
    \newtheorem{example}[figure]{Example}

\theoremstyle{remark}
        \newtheorem{remark}[figure]{Remark}

    \newtheorem{ack}[figure]{Acknowledgements}

\newcommand{\sect}{\setcounter{figure}{0}\section}

\renewcommand{\thefigure}{\arabic{section}.\arabic{figure}}

\newenvironment{myeq}[1][]
{\stepcounter{figure}\begin{equation}\tag{\thefigure}{#1}}
{\end{equation}}

\newcommand{\ZZ}{{\mathbb Z}}

\newcommand{\Aut}{\operatorname{Aut}}
\newcommand{\conn}{\operatorname{conn}}
\newcommand{\Ext}{\operatorname{Ext}}

\newcommand{\Hom}{\operatorname{Hom}}
\newcommand{\Id}{\operatorname{Id}}

\newcommand{\A}{\mathcal{A}_{\#}}
\renewcommand{\AA}{\mathcal{A}_{\ast}}
\newcommand{\AAA}{\mathcal{A}^{\ast}}
\newcommand{\NA}{N\A}
\newcommand{\NAA}{N\AA}
\newcommand{\NAAA}{N\AAA}

\begin{document}

\title{Homology self closeness number and cofibration sequence}
\author{Gopal Chandra Dutta}
\address{Department of Mathematics and Statistics\\
		Indian Institute of Technology, Kanpur\\ Uttar Pradesh 208016\\India}
	
\email{gopald@iitk.ac.in}

\date{\today}


\begin{abstract}
We study the homology self-closeness numbers of simply connected CW complex and those of the homotopy cofiber. Self-maps of spaces in cofibrations which appear in Homology decomposition are studied. We also consider Postnikov tower and studied the relation of homology self-closeness numbers between homology decomposition and homotopy decomposition. 
\end{abstract}
\maketitle

\maketitle

\sect{Introduction}
For a pointed space $X$, the set of homotopy classes of base point preserving maps $X\to X$ is denoted by  $[X,X]$. This is a monoid under composition of maps. The subset of all invertible elements of this monoid is denoted by $\Aut(X)$. It is the group of homotopy classes of self-equivalences on $X$. For notational convenience, we do not distinguish between a map $f \colon X\to X$ and that of its homotopy class in $[X, X]$. The group $\Aut(X)$ has been studied by several authors (cf. \cite{AGSE, RGSHD}). Choi and Lee introduced the following sub-monoids in \cite{CLCNGS}:
$$\A^k(X): = \big\{f\in [X,X]: ~f_{\#}\colon \pi_i(X)\xrightarrow{ \cong } \pi_i(X), \text{ for all } 0\leq i\leq k \big\}.$$ 
For connected CW complex $X$, we have the chain: $$[X,X] = \A^0(X) \cdots \supseteq \A^k(X) \supseteq \A^{k+1}(X) \cdots \supseteq \A^{\infty}(X) = \Aut(X).$$
They introduced a numeric homotopy invariant associated to the chain, called homotopy  self-closeness number $\NA(X)$ defined as $$\NA(X): = \min \big\{k\geq 0: ~\A^k(X) = \Aut(X) \big\}.$$
Similar kind of monoids $\AA^k(X)$ and $\AAA_k(X)$ for homology and cohomology were introduced in \cite{OYSCFC}.
$$\AA^k(X): = \big\{f\in [X,X]: ~f_{\ast}\colon H_i(X)\xrightarrow{ \cong } H_i(X), \text{ for all } 0\leq i\leq k \big\};$$
$$\AAA_k(X): = \big\{f\in [X,X]: ~f^{\ast}\colon H^i(X)\xrightarrow{ \cong } H^i(X), \text{ for all } 0\leq i\leq k \big\}.$$
We know that for a simply connected CW-complex, a self-map is a homotopy equivalence if and only if it is a homology equivalence (or a cohomology equivalence). Therefore, we have the two types of chain for simply connected CW complex $$[X,X] = \AA^0(X) \cdots \supseteq \AA^k(X) \supseteq \AA^{k+1}(X) \cdots \supseteq \AA^{\infty}(X) = \Aut(X),$$ and 
$$[X,X] = \AAA_0(X) \cdots \supseteq \AAA_k(X) \supseteq \AAA_{k+1}(X) \cdots \supseteq \AAA_{\infty}(X) = \Aut(X).$$ 
This motivates to define \emph{homology and cohomology self-closeness numbers} (cf. \cite{OYSCFC}) 
$$\NAA(X): = \min \big\{k\geq 0: ~\AA^k(X) = \Aut(X) \big\},$$ 
$$\NAAA(X): = \min \big\{k\geq 0: ~\AAA_k(X) = \Aut(X) \big\}.$$
In \cite{OYSF}, Oda and Yamaguchi studied the homotopy self-closeness number of the following type of a \emph{fibration sequence}: 
\begin{equation}\label{eq1}
 \cdots \to K(G,n+1)\to X\to Y\xrightarrow{\gamma} K(G,n+2),
\end{equation}
where $G$ is an abelian group and $K(G,n)$ is the Eilenberg-MacLane space of type $(G,n)$. Note that $X$ is the homotopy fiber of the map $\gamma$. 
 
In 2021, Li studied the cohomology self-closeness number of the following type of \emph{cofibration sequence}:
\begin{equation}\label{LICS}
A\xrightarrow{\gamma} X\xrightarrow{\iota}Y\to \Sigma{A}.
\end{equation}
He obtained the relations between cohomology self-closeness numbers of $X$ and $Y$ (cf. \cite{LCHSCS}).
 
Recall that for an abelian group $G$ and integer $n\geq 2$, the Moore space $M(G,n)$ of type $(G,n)$ is the simply connected CW-complex $M(G,n)$ unique upto homotopy, such that 

$$
\tilde{H}_i(M(G,n)) = \begin{cases}
	G &\text{if} ~i=n, \\
	0 &\text{if} ~ i\neq n.
\end{cases} 
$$
If $G$ is free-abelian, then $M(G,n)$ is just the wedge of copies of $S^n$ and $M(\ZZ_m, n) = S^n\cup_{m} e^{n+1},$ where $(n+1)$-cell is attached to $S^n$ by a degree $m$ map. Thus, for a finitely generated $G$, Moore space $M(G,n)$ is a finite CW-complex of dimension $n$ if $G$ is free-abelian and of dimension $n+1$ if $G$ has torsion.
Note that $\pi_n(X; G) := [M(G,n), X]$ is a group for $n\geq 2$ and abelian group for $n\geq 3$. It is called the \emph{$n$-th homotopy group of $X$ with coefficients in $G$}. 

In this paper, we consider the dual \emph{cofibration sequence} of Equation \ref{eq1}:

\begin{equation}\label{eq2}
M(G,n)\xrightarrow{\gamma} X\xrightarrow{\iota}Y\to M(G,n+1)\to \cdots.
\end{equation}
Here $n\geq 2$, $Y$ is the homotopy cofiber of the map $\gamma$ and $\iota$ is a cofibration. This sequence is a particular case of the cofibration sequence \ref{LICS}. The group $\Aut(Y)$ was studied by Oka, Sawashita and Sugawara in \cite{OSSGSM}. The homotopy self-closeness number corresponding to the cofibration sequence \ref{eq2} was studied in \cite{OYSEC} by considering $G = \mathbb{Z}$.

In Section \ref{sec2}, we establish relations between $\NAA(X)$ and $\NAA(Y)$. Indeed, we find $\NAA(Y)$ by utilizing the given map $\gamma$ and making certain assumptions about $X$.
First, we prove the inequality $\NAA(X)\leq \NAA(Y)$, when the induced homomorphism $\gamma_{\#}\colon \pi_n(M(G,n);G) \to\pi_n(X;G)$ is an epimorphism (see Theorem \ref{ithm1}).

Moreover, the reverse inequality $\NAA(Y)\leq \NAA(X)$ holds if the induced homomorphism $\gamma_{\#}\colon \pi_n(M(G,n);G) \to \pi_n(X;G)$ is a monomorphism and  $\gamma^{\#}\big(\Aut(X)\big) \subseteq\gamma_{\#}\big(\pi_n(M(G,n);G)\big)$  (see Theorem \ref{ithm2}).

Further, we get the equality $\NAA(X) = \NAA(Y)$ if the induced homomorphism $\gamma_{\#}\colon \pi_n(M(G,n);G)\to \pi_n(X;G)$ is an isomorphism (see Theorem \ref{ehsn}).

As a special case, assume that $\pi_n(M(G,n);G)\cong \mathbb{Z}/{q\mathbb{Z}}\langle \Id_M \rangle$ and $\pi_n(X,G) \cong \mathbb{Z}/{m\mathbb{Z}}\langle \gamma \rangle$, where $q,m \geq 0$. Then
\begin{enumerate}[{(a)}]
\item $\NAA(X)\leq \NAA(Y), ~~ \text{if}~q = 0\neq m~ \text{or}~ q>m.$
\item $\NAA(X) = \NAA(Y)$ if $q = m$ and $\conn(X)\geq 2~ (\text{or}~ G~\text{is free})$.
\end{enumerate}
(see Proposition \ref{prop1}).

By taking $G = \mathbb{Z}/{p \mathbb{Z}}$ for a prime $p\neq2$, we obtain the relation $\NAA(Y)\leq \NAA(X)$ under some conditions (see Proposition \ref{prop2}).

In Section \ref{sec3}, we consider the \emph{homology decomposition} $\{X_n\}$ of a simply connected CW-complex $X$, where $X_n$ is the $n$-th \emph{homology section} (cf. \cite[Section 7.3]{AIH}). In Theorem \ref{rshsn}, we obtain the relations between homology self-closeness numbers of consecutive homology sections. Moreover, in Lemma \ref{egshe}, we prove the relationship between the group of self-homotopy equivalences of $X$ and $X_n$. Using this Lemma, we obtain the relation between the two numbers $\NAA(X)$ and $\NAA(X_n)$ under some conditions, see Theorem \ref{echd}.

In Section \ref{sec4}, we consider the \emph{homotopy decomposition (Postnikov tower)} $\{X^{(n)}\}$ of a simply connected CW-complex $X$, where $X^{(n)}$ is the $n$-th \emph{homotopy section}. The main aim of this section is to obtain relations between homology and homotopy self-closeness numbers and their corresponding decomposition. In Lemma \ref{ehhsn}, we obtain a relation between the homology self-closeness numbers of homotopy section and homology section. Theorem \ref{hsnchs} gives the analogous results of Theorem \ref{rshsn} for homotopy sections. Choi and Lee have studied the relation of homotopy self-closeness numbers of $X$ and $X^{(n)}$ \cite[Theorem 3.5]{CLSCWHD}. We find the homological version of the result in Theorem \ref{hshd}. We prove the equality $\NAA(X^{(n)}) = \NAA(X_n)$ under some conditions in Corollary \ref{ehshh}. Finally, we compute homology self-closeness numbers of homotopy sections (Postnikov tower) for some CW-complexes by the help of homology decomposition.

\begin{ack}
The author would like to thank IIT Kanpur for Ph.D fellowship.
\end{ack}

\sect{Homology self-closeness number}\label{sec2}
In this section, we study the relation between the homology self-closeness numbers of $X,Y$ for the cofibration sequence of Equation \ref{eq2}.
Recall that the \emph{connectivity} of a space is defined as $$\conn(X):= \min \big\{k\geq 0~:~ \pi_{k+1}(X)\neq 0\big\}.$$ Thus, $\pi_k(X) = 0$ for all $k\leq \conn(X)$. Moreover, we denote the \emph{homological dimension} by
$$H_{\ast}\text{-}\dim(X):= \max\big\{k\geq 0~:~ H_k(X)\neq 0\big\}.$$ Therefore, $H_k(X) = 0$ for all $k> H_{\ast} \text{-}\dim(X)$. For any simply connected CW-complex $X$, we have $$2 = \conn(X) + 1\leq \NAA(X)\leq H_{\ast} \text{-} \dim(X).$$

\noindent Recall that, a map $f\colon X\to X'$ is called \emph{homotopical $n$-equivalence} if $f_{\#}\colon \pi_k(X)\to \pi_k(X')$ is isomorphism for all $k<n$ and an epimorphism for $k=n$. A map $f\colon X\to X'$ is called \emph{homological $n$-equivalence} if $f_{\ast}\colon H_k(X)\to H_k(X')$ is isomorphism for all $k<n$ and an epimorphism for $k=n$.

If $X'$ is simply connected then the mapping cone $C_f$ of a map $f\colon X\to X'$ is simply connected (cf \cite[Proposition B.4 ]{AIH}).

\begin{lemma}\label{leq}
Let $f\colon X\to X'$ be a map between simply connected spaces. Then the following are equivalent:
\begin{enumerate}[{(a)}]
	\item $f$ is homotopical $n$-equivalence.
	\item $f$ is homological $n$-equivalence.
	\item $C_f$ is $n$-connected.
\end{enumerate}
\end{lemma}

\begin{proof}
Follows from \cite[Theorem 4]{JHCWCH} and \cite[Lemma 6.4.12]{AIH}.

\end{proof}

From now onwards, $X$ will denote a simply connected CW-complex and $G$ a finitely generated abelian group.





\begin{theorem}\label{ithm1}
$M(G,n)\xrightarrow{\gamma} X \xrightarrow{\iota} Y$ be a cofibration sequence such that the induced homomorphism $\gamma_{\#}\colon \pi_n(M(G,n);G)\to \pi_n(X; G)$ is an epimorphism. Then
\begin{enumerate}[{(a)}]
\item  $2\leq \NAA(X)\leq \NAA(Y)\leq n+1$ if $H_{\ast}\text{-}\dim(X)< n$.
\item $n\leq \NAA(Y)$ if $H_n(X) = H_{n+1}(X) = 0$ and $2\leq n\leq \NAA(X)$.
\end{enumerate}
\end{theorem}

\begin{proof}
\begin{enumerate}[{(a)}]\setlength\itemsep{1.5 em}
\item Note that $\NAA(X)\leq H_{\ast}\text{-}\dim(X)< n.$ Moreover, $H_k(Y) = 0$ for all $k\geq n+2$ using the long exact sequence of homology groups of the cofibration sequence. Therefore, $\NAA(Y)\leq n+1$.
Assume that $\NAA(Y) = m$ and $g\in \AA^m(X)$. Let $m< n$, otherwise we get the desired result.

Since $\gamma_{\#}$ is epimorphism, so there exists $h\colon M(G,n)\to M(G,n)$ such that $g\circ \gamma = \gamma \circ h$. Therefore, there is a map $f\in [Y,Y]$ such that $f\circ \iota = \iota \circ g$. Hence, we get a homotopy commutative diagram 

\begin{align}\label{dia1}
	\xymatrix{
	M(G,n) \ar[dd]_{h}  \ar[rr]^{\gamma}  &&  X \ar[dd]^{g} \ar[rr]^{\iota} && Y \ar[dd]^{f} \\\\
	M(G,n) \ar[rr]^{\gamma} && X \ar[rr]^{\iota} && Y
    }
\end{align}
Note that $\iota_{\ast}\colon H_k(X)\to H_k(Y)$ is isomorphism for all $k\leq n-1$. Therefore, $f_{\ast}\colon H_k(Y)\to H_k(Y)$ is isomorphism for all $k\leq m$. So $f\in \AA^m(Y) = \Aut(Y)$. Hence, using the commutativity of the right side diagram, we have $g\in \AA^{n-1}(X) = \Aut(X)$. Consequently, $\NAA(X)\leq m = \NAA(Y)$ and we get the desired result.
\item Let $\NAA(Y) = m$ and $g\in \AA^m(X)$. As in the proof of part (a), we have the diagram \ref{dia1}. If possible, assume that $m < n$. Therefore, $f\in \AA^m(Y) = \Aut(Y)$. Observe the following commutative diagram: 

\begin{align}
	\xymatrix{
		H_{n+1}(Y) \ar[dd]_{f_{\ast}}  \ar[rr]^{\cong}  &&  H_n(M(G,n)) \ar[dd]^{h_{\ast}} \\\\
		H_{n+1}(Y) \ar[rr]^{\cong} && H_n(M(G,n))
	}
\end{align}
Thus $h\in \Aut(M(G,n))$. Therefore, using the five Lemma on the long exact sequence of homology groups of the diagram \ref{dia1}, we have $g\in \Aut(X)$. This contradicts the fact that $\NAA(X)\geq n$. Hence, $\NAA(Y) = m \geq n$.
\end{enumerate}
\end{proof}

\begin{remark}
If $\gamma = 0$, in that case $Y = X\vee M(G,n+1)$ for $n\geq 2$. This follows that the difference between $\NAA(X)$ and $\NAA(Y)$ only depends on $X$. Let $X = S^2$, then $\NAA(X) = 2< n+1 = \NAA(Y).$
\end{remark}

\begin{theorem}\label{ithm2}
$M(G,n)\xrightarrow{\gamma} X\xrightarrow{\iota} Y$ be a cofibration sequence such that the induced homomorphism $\gamma_{\#}\colon \pi_n(M(G,n);G)\to \pi_n(X;G)$ is a monomorphism and $\gamma^{\#}\big(\Aut(X)\big)$  $\subseteq\gamma_{\#}\big(\pi_n(M(G,n);G)\big)$. Then
\begin{enumerate}[{(a)}]
	\item $3\leq \NAA(Y)\leq \NAA(X)\leq n-1$ if $X$ is $2$-connected, $H_{\ast}\text{-}\dim(X)< n$ and $G$ has torsion. 
	\item $2\leq \NAA(Y)\leq \NAA(X)\leq n-1$ if $H_{\ast}\text{-}\dim(X)< n$ and $G$ is free.
\end{enumerate}
\end{theorem}

\begin{proof}
\begin{enumerate}[{(a)}]
\item Observe that $\iota\colon X\to Y$ is homological $n$-equivalence. So homotopical $n$-equivalence by Lemma \ref{leq}. Since $H_{\ast}\text{-}\dim(X)<n$, therefore using \cite[Proposition 4C.1]{AHAT} there exits a CW-complex say $Z$ which is homotopy equivalence to $X$ and $\dim(Z)\leq n$. Thus the induced map $\iota_{\#}\colon [Z,X]\to [Z,Y]$ is onto. Moreover, let $p\colon Z\to X$ is a homotpy equivalence map with homotopy inverse $q\colon X\to Z$. Therefore, $p^{\#}\colon [X,X]\to [Z,X]$ is a bijection  map defined as $f\mapsto f\circ p$. Similarly, $q^{\#}\colon [Z,Y]\to [X,Y]$ is bijection. Consequently, the induced map $\iota_{\#}\colon [X,X]\to [X,Y]$ is surjective.	
	
Let $\NAA(X) = l<n$ and $f\in \AA^l(Y)$. By surjectivity of $\iota_{\#}\colon [X,X]\to [X,Y]$ implies that there exist some $g\in [X,X]$ such that $\iota\circ g = f\circ \iota$. Note that $\conn(X) = 2$ implies that $\conn(Y) = 2$. Therefore, the homotopical connectivity of the map $\gamma \colon M(G,n)\to X$ is 2 by Lemma \ref{leq}. Therefore, using \cite[Proposition 4.4]{RMEFCC} there exists a map $h\colon M(G,n)\to M(G,n)$ such that we get the homotopy commutative diagram \ref{dia1}.
Since $\iota_{\ast}\colon H_k(X)\xrightarrow{\cong} H_k(Y)$ for all $k\leq l$, therefore $g\in \AA^l(X)= \Aut(X)$. Let $\bar{g}\in \Aut(X)$ such that $g\circ \bar{g} = \bar{g}\circ g = \Id_X$. It is sufficient to show that $h\in \Aut(M(G,n))$.

From the assumption (ii), $\gamma^{\#}(\bar{g})\in \gamma_{\#}\big([M(G,n), M(G,n)]\big)$. Then there exits $\bar{h}\in [M(G,n), M(G,n)]$ such that $\gamma_{\#}(\bar{h}) = \gamma^{\#}(\bar{g})$. So $\gamma\circ \bar{h} = \bar{g}\circ \gamma$.  Therefore, we have $$\gamma \circ \bar{h}\circ h = \bar{g}\circ \gamma \circ h = \bar{g}\circ g\circ \gamma = \gamma,$$ and similarly, $\gamma \circ h\circ \bar{h} = \gamma$. Injectivity of $\gamma_{\#}$ implies that  $h\circ \bar{h} = \bar{h}\circ h = \Id_{M(G,n)}$. So $h\in \Aut \big(M(G,n)\big)$. Using  five Lemma on the long exact sequence of homology groups of the Diagram \ref{dia1}, we have $f\in \Aut(Y)$. Consequently, $$\NAA(Y)\leq l = \NAA(X).$$
\item Note that $\conn(X) = 1$ implies that $\conn(Y) =1$. Moreover, if $G$ is free abelian group, then $\dim(M(G,n)) = n$. Therefore, we can use \cite[Proposition 4.4]{RMEFCC} to obtain the diagram \ref{dia1}. Rest of the proof follows similarly.
\end{enumerate}
\end{proof}

\begin{theorem}\label{ehsn}
$M(G,n)\xrightarrow{\gamma} X\xrightarrow{\iota} Y$ be a cofibration sequence such that the induced homomorphism $\gamma_{\#}\colon \pi_n(M(G,n);G)\to \pi_n(X;G)$ is an isomorphism. Then
\begin{enumerate}[{(a)}]
	\item $\NAA(Y) = \NAA(X)$ if $X$ is $2$-connected, $H_{\ast}\text{-}\dim(X)< n$ and $G$ has torsion. 
	\item $\NAA(Y) = \NAA(X)$ if $H_{\ast}\text{-}\dim(X)< n$ and $G$ is free.
\end{enumerate}
\end{theorem}

\begin{proof}
From Theorem \ref{ithm1}, we have $\NAA(X)\leq \NAA(Y)$. 

For the converse part, surjectivity of the map $\gamma_{\#}\colon \pi_n(M(G,n);G)\to \pi_n(X;G)$ implies that $$\gamma^{\#}\big(\Aut(X)\big)\subseteq \pi_n(X;G) = \gamma_{\#}\big(\pi_n(M(G,n);G)\big).$$
Therefore, using Theorem \ref{ithm2}, we have $\NAA(Y)\leq \NAA(X)$. Hence, we get the desired result. 

\end{proof}


\begin{proposition}\label{prop1}
$M(G,n)\xrightarrow{\gamma} X\xrightarrow{\iota} Y$ be a cofibration sequence such that $H_{\ast}\text{-}\dim(X)<n$. Assume that there exist non-negative integers $q,m$ such that $$\pi_n(M(G,n);G)\cong \mathbb{Z}/{q\mathbb{Z}}\langle \Id_M \rangle,~~\pi_n(X,G) \cong \mathbb{Z}/{m\mathbb{Z}}\langle \gamma \rangle.$$ Then 
\begin{enumerate}[{(a)}]
	\item $\NAA(X)\leq \NAA(Y)$ if $q=0\neq m \text{ or } ~q>m$.
	\item $\NAA(X) = \NAA(Y)$ if $q = m$, $X$ is $2$-connected and $G$ has torsion.
	\item $\NAA(X) = \NAA(Y)$ if $q = m$ and $G$ is a free abelian group.
\end{enumerate}
\end{proposition}

\begin{proof}
\begin{enumerate}[{(a)}]
	\item  Note that $\gamma_{\#}(f_1+f_2) = \gamma\circ (f_1+f_2)\cong \gamma\circ f_1 + \gamma \circ f_2$. So $\gamma_{\#}$ is a homomorphism. It is sufficient to show that $\gamma_{\#}$ is surjective. Let $g\in \pi_n(X;G)$. Then there exists an integer $s$ such that $g = s\gamma$. Observe that $$g = s\gamma = \gamma \circ s\Id_M = \gamma_{\#}(s\Id_M).$$ Thus $\gamma_{\#}$ is surjective. Therefore using Theorem \ref{ithm1}(a), we get the desired result.
	\item Assume that $q = m$. As in the proof of part(a), we have surjectivity of the homomorphism $\gamma_{\#}$. Moreover, $\gamma_{\#}$ is injective as  $q = m$. Therefore, we get the equality using Theorem \ref{ehsn}(a).
	\item Similar as part (b).
\end{enumerate}
\end{proof}

The following Corollary is a homological version of \cite[Theorem 4]{OYSEC}.  
\begin{corollary}\label{hves}
Let $X$ be a simply connected CW-complex such that $\dim(X)\leq n-1$ and $\pi_n(X)\cong \mathbb{Z}$, where $n\geq 2$. If $\gamma\colon S^n\to X$ is a generator of $\pi_n(X)$, then $$\NAA(X) = \NAA(X\cup_{\gamma} e^{n+1}).$$
\end{corollary}

\begin{proof}
It is enough to show that $\gamma_{\#}\colon [S^n,S^n]\to [S^n,X]$ is a bijection. For surjectivity, let $g\in [S^n,X].$ Therefore there exists an integer $s$ such that $g = s\gamma$. We know that $s\gamma = \gamma \circ s\Id_{S^n}$, where $\Id_{S^n}$ is a generator of $\pi_n(S^n)$. Thus $g = \gamma_{\#}(s\Id_{S^n})$. Hence $\gamma_{\#}$ is surjective.

For injectivity, let $f_1, f_2\in [S^n,S^n]$ such that $\gamma_{\#}(f_1) = \gamma_{\#}(f_2)$. This implies $\gamma \circ f_1 = \gamma \circ f_2$. Moreover observe that $f_1 = k_1\Id_{S^n},~~ f_2 = k_2\Id_{S^n}$ for some $k_1,k_2\in \mathbb{Z}$. Therefore $\gamma \circ k_1\Id_{S^n} = \gamma \circ k_2\Id_{S^n}$. This implies $k_1\gamma = k_2\gamma$. So that $k_1 = k_2$. Then using Theorem \ref{ehsn} we get the desired result.

\end{proof}

\begin{remark} Corollary \ref{hves} can also be deduced directly from Proposition \ref{prop1}.
\end{remark}

Recall the generalised Fruedenthal suspension theorem: Let $Y$ be an $n$-connected CW-complex and $X$ is finite dimensional CW-complex. Then the suspension map $\Sigma\colon [X,Y]\to [\Sigma X,\Sigma Y]$ is isomorphism if $\dim(X)< 2n+1$ and epimorphism if $\dim(X) = 2n+1$ (\cite[Theorem 1.21]{CSH}).

\begin{proposition}\label{}
Consider the Hopf map $S^3\xrightarrow{\eta_2} S^2$ and all its suspension map $S^{n+1}\xrightarrow{\eta_n} S^n$, where $n\geq 3$. Then 

\[\NAA(S^n\cup_{\eta_n}e^{n+2}) = \begin{cases}
	2 &\text{if} ~n=2, \\
	n+2 &\text{if}~n\geq 3.
\end{cases} \] 

\end{proposition}

\begin{proof}
For $n=2$, the Hopf map is $S^3\xrightarrow{\eta_2} S^2$  generates the group $\pi_3(S^2)\cong \mathbb{Z}$. Therefore, using Corollary \ref{hves}, we have $\NAA(S^2) = \NAA(S^2\cup_{\eta_2} e^4)$. Hence $\NAA(S^2\cup_{\eta_2} e^4) = 2$.

For $n\geq 3,$ the suspension of $\eta_2$ is $S^{n+1}\xrightarrow{\eta_n} S^n$. Using \cite[Corollary 4J.4]{AHAT}, we have $\pi_{n+1}(S^n) \cong \mathbb{Z}_2\langle \eta_n \rangle$. Therefore, we have $\NAA(S^n)\leq \NAA(S^n \cup_{\eta_n} e^{n+2})$ by Proposition \ref{prop1}(a). Hence, $$n\leq \NAA(S^n \cup_{\eta_n} e^{n+2})\leq n+2, ~\text{ for all } n\geq 3.$$
It is sufficient to show that $\NAA(S^n\cup_{\eta_n}e^{n+2})\neq n$. Let $P = S^n\cup_{\eta_n}e^{n+2}$. By \cite[Section 8]{ATMSMC}, we have $$[P,P]\cong \mathbb{Z}\langle \Id_P \rangle \oplus \mathbb{Z}\langle\iota\circ \bar{\xi}\rangle,$$ 
where $\Id_P\colon P\to P$ is the identity map, $\iota\colon S^n\to P$ is inclusion map, $q\colon P\to S^{n+2}$ is quotient map and $\bar{\xi}\in [P, S^n],~~ \widetilde{\xi}\in [S^{n+2},P]$ satisfying the following relations 

\begin{equation}\label{hopf}
\bar{\xi}\circ \iota = 2\cdot \Id_{S^n}, ~~ q\circ \widetilde{\xi} = 2\cdot \Id_{S^{n+2}}, ~~~ \iota\circ \bar{\xi} + \widetilde{\xi}\circ q = 2\cdot \Id_p. 
\end{equation}
Observe that the homomorphisms $\iota_{\ast}\colon H_n(S^n)\xrightarrow{\cong} H_n(P)$ and $q_{\ast}\colon H_{n+2}(P)\xrightarrow{\cong} H_{n+2}(S^{n+2})$ are isomorphisms.
Let $H_0(S^0)\cong \ZZ \langle a_0 \rangle.$ Consider the suspension map $H_0(S^0)\xrightarrow{\Sigma^k} H_k(S^k)$ defined as $a_0\mapsto a_k=\Sigma^k(a_0)$, which is an isomorphism for each $k\geq 1$. Observe that 

\[H_k(P) = \begin{cases}
	\mathbb{Z}\langle b \rangle &\text{if} ~k = n, \\
	\mathbb{Z}\langle c \rangle  &\text{if}~~k = n+2,\\
	0 &\text{ otherwise },
\end{cases} \] 
where $b = \iota_{\ast}(a_n)$, and  $c= (q_{\ast})^{-1}(a_{n+2})$. From the relation \ref{hopf}, we have
$$(\iota\circ \bar{\xi})_{\ast}(b) = (\iota\circ \bar{\xi}\circ \iota)_{\ast}(a_n) = (2\cdot \iota\circ \Id_{S^n})_{\ast}(a_n) = 2\cdot \iota_{\ast}(a_n) = 2\cdot b.$$
Moreover, $\widetilde{\xi}_{\ast}(a_{n+2}) = (q_{\ast})^{-1}\circ q_{\ast}\circ \widetilde{\xi}_{\ast}(a_{n+2}) = 2\cdot (q_{\ast})^{-1}(a_{n+2}) = 2\cdot c.$

\noindent Therefore,
$$(\iota\circ \bar{\xi})_{\ast}(c) = (2\cdot \Id_P - \widetilde{\xi}\circ q)_{\ast}(c) = 2\cdot c - \widetilde{\xi}_{\ast}\circ q_{\ast}(c) = 2\cdot c - \widetilde{\xi}_{\ast}(a_{n+2}) = 0.$$

So for $f\in \AA^n(P)$, we have $f_{\ast}\colon H_k(P)\xrightarrow{\cong} H_k(P)$ for all $k\leq n$. Note that $$f = m\Id_P + r(\iota\circ \bar{\xi}), \text{ where } m,r\in \mathbb{Z}.$$ Thus $f_{\ast}(b) = mb + 2rb = (m+2r)b$. This implies that $m+2r = +1 \text{ or } -1$. Therefore, $f = 3\Id_P - (\iota\circ \bar{\xi})\in \AA^n(P)$. But $f\notin \Aut(P)$ as $f_{\ast}(c) = 3c.$ Hence $\NAA(P)\neq n.$

\end{proof}

\begin{remark}
Note that $\Sigma(C_{\eta_n}) = C_{\Sigma \eta_n} = C_{\eta_{n+1}}$ for all $n\geq 2$. Therefore, $\Sigma^{n-2} \mathbb{C}P^2 = S^n\cup_{\eta_n} e^{n+2}$ for all $n\geq 2$. Hence, 
\[\NAA(\Sigma^{n-2}\mathbb{C}P^2) = \begin{cases}
	2 &\text{if} ~n =2, \\
	n+2 &\text{if}~n\geq 3.
\end{cases} \] 
\end{remark}

\begin{proposition}\label{prop2}
$M(G,n)\xrightarrow{\gamma} X\xrightarrow{\iota} Y$ be a cofibration sequence such that $X$ is $2$-connected and $H_{\ast}\text{-}\dim(X)<n$. Assume that any one of the following conditions hold:
\begin{enumerate}[{(i)}]
\item For $q\geq 0, ~\pi_n(M(G,n);G)\cong \mathbb{Z}/{q\mathbb{Z}}\langle \Id_M\rangle$ and $\gamma$ is a generator of any $\mathbb{Z}$ direct summand of the abelian group $\pi_n(X;G)$.
\item For a prime $p>2$, if $G = \mathbb{Z}/{p\mathbb{Z}}$ and $\gamma$ is a generator of any $\mathbb{Z}$ direct summand of the abelian group $\pi_n(X;G)$.
\end{enumerate} 
Then $\NAA(Y)\leq \NAA(X)$.
\end{proposition}

\begin{proof}
Let $\NAA(X) = l<n$ and $f\in \AA^l(Y)$. As in the proof of Theorem \ref{ithm2}, we have $g\colon X\to X$ and $h\colon M(G,n)\to M(G,n)$, which satisfies the commutative diagram \ref{dia1}. Moreover, $g\in \AA^l(X) = \Aut(X)$. Let $\bar{g}$ be the homotopy inverse of $g$. It is sufficient to show that $h\in \Aut(M(G,n))$.

\begin{enumerate}[{(i)}] 
\item Let $\pi_n(X,G)\cong \mathbb{Z}\langle \gamma\rangle \oplus U$ for some subgroup $U$. Then we have $h = s\Id_M$ for some $s\in \mathbb{Z}$. Moreover $\bar{g}\circ \gamma = t\gamma + u$ for some $t\in \mathbb{Z}$ and some $u\in U$. It follows that $$\gamma = \bar{g}\circ g\circ \gamma = \bar{g}\circ \gamma\circ h = (t\gamma + u)\circ s\Id_M = ts\gamma + su.$$ This implies that $ts = 1$ and $su = 0$. So $s = +1\text{ or}-1$, hence $h\in \Aut(M(G,n)).$

\item From  universal coefficient theorem for homotopy, we have 
$$0\to \Ext \big(G,\pi_{n+1}(Z)\big)\to \pi_n(Z,G)\to \Hom\big(G,\pi_n(Z)\big)\to 0.$$
Take $Z = M(\mathbb{Z}/{p\mathbb{Z}},n)$. From \cite[Section 1]{AMSEHCH}, we have $$\pi_{n+1}(Z) = \pi_{n+1}(M(\mathbb{Z}/{p\mathbb{Z}},n)) = \mathbb{Z}/{p\mathbb{Z}}\otimes \mathbb{Z}/{2\mathbb{Z}} = 0.$$ 
Therefore, $\pi_n(M(\mathbb{Z}/{p\mathbb{Z}},n); \mathbb{Z}/{p\mathbb{Z}})\cong \Hom(\mathbb{Z}/{p\mathbb{Z}},\mathbb{Z}/{p\mathbb{Z}})\cong \mathbb{Z}/{p\mathbb{Z}}\langle \Id_M \rangle$. Hence, using part (i), we get the desired result. 

\end{enumerate}
\end{proof}

\begin{remark}
If $G$ is a free abelian group, then simply connectedness is sufficient for Proposition \ref{prop2}.
\end{remark}	

The following Corollary is a homological version of \cite[Theorem 5]{OYSCFC}.

\begin{corollary}\label{iscn}
Let $X$ be a simply connected CW-complex such that $\dim(X)\leq n-1$, where $n\geq 2$. If $\gamma\colon S^n\to X$ is a generator of a direct summand $\mathbb{Z}\subset\pi_n(X)$, then $$\NAA(X\cup_{\gamma} e^{n+1})\leq \NAA(X).$$
\end{corollary}

\begin{proof}
Follows from Proposition \ref{prop2}(i), by taking $G = \mathbb{Z}$.

\end{proof}

\begin{example}
Finiteness theorem of Serre \cite[Section 6]{SGHCG} $$\pi_{4n-1}(S^{2n})\cong \mathbb{Z}\oplus F_n,  \text{ for all } n\geq 1,$$ where $F_n$ is a finite abelian group. Let $\gamma$ be a generator of the $\mathbb{Z}$ direct summand of the abelian group $\pi_{4n-1}(S^{2n})$. Therefore, $\NAA(S^{2n}\cup_{\gamma} e^{4n})\leq \NAA(S^{2n})$ by using Corollary \ref{iscn}. Hence, $$\NAA(S^{2n}\cup_{\gamma} e^{4n})\leq 2n.$$
Observe that $\AA^{2n-1}(S^{2n}\cup_{\gamma}e^{4n})(X) = [S^{2n}\cup_{\gamma}e^{4n}~ , ~S^{2n}\cup_{\gamma}e^{4n}]$. Let $f\colon S^{2n}\cup_{\gamma}e^{4n}\to S^{2n}\cup_{\gamma}e^{4n}$ be a constant map at base point. Therefore, $f\in [S^{2n}\cup_{\gamma}e^{4n}~,~S^{2n}\cup_{\gamma}e^{4n}]$ but $f\notin \Aut(S^{2n}\cup_{\gamma}e^{4n})$. Hence,
$$\NAA(S^{2n}\cup_{\gamma} e^{4n}) = \NAA(S^{2n}) = 2n.$$
\end{example}


\sect{Homology decomposition}\label{sec3}
	
In this section, we consider homology decomposition of a space and its relation with
homology self-closeness numbers. First, we recall the definition.
\begin{definition}
Let $X$ be a simply connected CW-complex. A \emph{homology decomposition} of $X$ consists of a sequence of simply connected CW-complexes $\{X_n\}$ and structure maps 
$$j_n\colon X_n\to X, \iota_n\colon X_n\to X_{n+1},~k_{n+1}:M\big(H_{n+1}(X),n\big)\to X_n,$$ that satisfy the following conditions:
\begin{enumerate}[{(i)}]
\item $j_{n\ast}\colon H_k(X_n)\to H_k(X)$ is an isomorphism for all $k\leq n$ and $H_k(X_n) = 0$ for all $k>n$.
\item $M\big(H_{n+1}(X),n\big)\xrightarrow{k_{n+1}}X_n\xrightarrow{\iota_n}X_{n+1}$ is a cofibration sequence of the cellular map $k_{n+1}$ such that the induced map $k_{n+1\ast}\colon H_n(M(H_{n+1}(X),n))\to H_n(X_n)$ is trivial.
\item The following diagram commutes.
$$
\xymatrix{
&&X_{n+1} \ar[ddll]_{j_{n+1}} \\\\
X  &&   X_n  \ar[ll]^{j_n} \ar[uu]_{\iota_n} 
}
$$
\end{enumerate}	
\end{definition}

\noindent The collection  $\big\{X_n,j_n,\iota_n,k_n\big\}$ is called a homology decomposition of $X$. The spaces $X_n$ are called the $n$-th homology sections of $X$ and the maps $k_{n+1}\colon M\big(H_{n+1}(X),n\big) \to X_n$ are called homological  $k$-invariants. It is well known that every simply connected CW-complex admits a homology decomposition (see \cite{AIH,AHAT}). Note that if all homology groups are free then the cellular skeletons serve as a homology decompositions, with the cell attaching maps as $k$-invariants.

\begin{remark}\label{}
It follows from the definition that $\NAA(X_n)\leq n.$ 
\end{remark}

The following Theorem deduce relations between the homology self-closeness numbers of two consecutive homology section.

\begin{theorem}\label{rshsn}
Let $X$ be a simply connected $CW$-complex. For a homology decomposition $\big\{X_n,j_n,\iota_n,k_n\big\}$ of $X$, we have the following:
\begin{enumerate}[{(a)}]\setlength\itemsep{1 em}
\item If the induced homomorphism
$$k_{m+1\#}\colon \pi_m \big(M(H_{m+1}(X),m);H_{m+1}(X)\big)\to \pi_m\big(X_m; H_{m+1}(X)\big)$$
is an epimorphism for some $m\in \mathbb{N}$, then $\NAA(X_m)\leq \NAA(X_{m+1})$.
\item If the induced homomorphism 
$$k_{m+1\#}\colon \pi_m \big(M(H_{m+1}(X),m);H_{m+1}(X)\big)\to \pi_m\big(X_m; H_{m+1}(X)\big)$$
is a monomorphism such that $$k^{\#}_{m+1}(\Aut(X_m))\subseteq k_{m+1\#}\Big(\pi_m \big(M(H_{m+1}(X),m);H_{m+1}(X)\big)\Big),$$ where $H_k(X)$ is free-abelian group for $k=m,m+1$, then $$\NAA(X_{m+1})\leq \NAA(X_m).$$
\item If $H_k(X)$ is free-abelian group for $k=m,m+1$ and the induced homomorphism $k_{m+1\#}$ is an isomorphism, then $\NAA(X_m) = \NAA(X_{m+1})$.
\end{enumerate}
\end{theorem}

\begin{proof}

\begin{enumerate}[{(a)}]
\item Let $\NAA(X_{m+1}) = l\leq m+1$ and $g\in \AA^l(X_m)$. If $l\geq m$, then $\AA^l(X_m) = \Aut(X_m)$ and we get the desired result.
	
Let us assume that $l<m$. Surjectivity of $k_{m+1\#}$ implies that there exists $h\colon M\big(H_{m+1}(X),m\big)\to M\big(H_{m+1}(X),m\big)$ such that $k_{m+1}\circ h = g\circ k_{m+1}$. Therefore, we get a homotopy commutative diagram 
    
\begin{align}\label{dia2}
  \xymatrix{
  M\big(H_{m+1}(X),m\big) \ar[dd]_{h}  \ar[rr]^{k_{m+1}}  &&  X_m \ar[dd]^{g} \ar[rr]^{\iota_m} && X_{m+1} \ar[dd]^{f} \\\\
  M\big(H_{m+1}(X),m\big) \ar[rr]^{k_{m+1}} && X_m \ar[rr]^{\iota_m} && X_{m+1}
           }
\end{align}
Observe that $\iota_{m\ast}\colon H_k(X_m)\xrightarrow{\cong} H_k(X_{m+1})$ for all $k\leq m$.  Therefore, $f\in \AA^l(X_{m+1}) = \Aut(X_{m+1})$. Using the commutativity of the right side diagram, we have $g\in \AA^m(X_m) = \Aut(X_m)$. Consequently, $$\NAA(X_m)\leq l = \NAA(X_{m+1}).$$
\item Observe that $H_{\ast} \text{- }\dim(X_m)\leq m$ and the map $\iota_m\colon X_m\to X_{m+1}$ is homotopical $m$-equivalence. So the induced map $\iota_{\#}\colon [X_m,X_m]\to [X_m,X_{m+1}]$ is surjective.
  
Let $\NAA(X_m) = r\leq m$ and $f\in \AA^r(X_{m+1})$. By surjectivity of $\iota_{\#}$ there exists $g\in [X_m,X_m]$ such that $\iota_m\circ g = f\circ \iota_m$. Therefore, using \cite[Proposition 4.4]{RMEFCC} there is a map $h\colon M\big(H_{m+1}(X),m\big)\to M\big(H_{m+1}(X),m\big)$ such that we have the homotopy commutative Diagram \ref{dia2}. Since $\iota_{m\ast}\colon H_k(X_m)\xrightarrow{\cong} H_k(X_{m+1})$ for all $k\leq m$, therefore $g\in \AA^r(X_m) = \Aut(X_m)$. Thus, there is a $\bar{g}\in \Aut(X_m)$ such that $g\circ \bar{g} = \Id_{X_{m}} = \bar{g}\circ g$. From the given assumption $$k^{\#}_{m+1}(\bar{g})\in k_{m+1\#}\big(\pi_m \big(M(H_{m+1}(X),m);H_{m+1}(X)\big)\big).$$ So there exists $\bar{h}\colon M\big(H_{m+1}(X),m\big)\to M\big(H_{m+1}(X),m\big)$ such that $k_{m+1\#}(h) = k^{\#}_{m+1}(\bar{g})$, i.e. $k_{m+1}\circ \bar{h} = \bar{g}\circ k_{m+1}$. Observe that $$k_{m+1}\circ \bar{h}\circ h = \bar{g}\circ k_{m+1}\circ h = \bar{g}\circ g\circ k_{m+1} = k_{m+1}.$$ Similarly, $k_{m+1}\circ h \circ \bar{h} = k_{m+1}.$ Therefore, injectivity of $k_{m+1\#}$ implies that $h\in \Aut\big(M\big(H_{m+1}(X),m\big)\big)$. Using  five Lemma on the long exact sequence of homology groups we have $f\in \Aut(X_{m+1})$. Consequently, $$\NAA(X_{m+1})\leq \NAA(X_m).$$
  
\item Note that $k^{\#}_{m+1}(\Aut(X_m))\subseteq \pi_m \big(X_m; H_{m+1}(X)\big)$.
Surjectivity of $k_{m+1\#}$ implies that $k_{m+1\#}\Big(\pi_m \big(M(H_{m+1}(X),m);H_{m+1}(X)\big)\Big) = \pi_m \big(X_m; H_{m+1}(X)\big)$.
Therefore using (a) and (b), we get the desired result.
\end{enumerate}	
\end{proof}

The following Proposition gives a relation between two consecutive homology sections which is independent of $k_{\#}$ and $k^{\#}$ maps.
\begin{proposition}\label{iihsn}
Let $X$ be a simply connected CW-complex. If $H_m(X)$ is finitely generated free abelian group for some $m\in \mathbb{N}$ then either $\NAA(X_{m+1}) = m+1$ or $\NAA(X_{m+1})\leq \NAA(X_m)$.
\end{proposition}

\begin{proof}
Observe that the map $\iota_m\colon X_m\to X_{m+1}$ is a homotopical $m$-equivalence. Then $\iota_{m\#}\colon [X_m, X_m]\to [X_m, X_{m+1}]$ is a surjective map. Assume that $\NAA(X_{m+1})< m+1.$ Let $\NAA(X_m) = l\leq m$ and $f\in \AA^l(X_{m+1})$. Then there exits $g\colon X_m \to X_m$ which make the following diagram homotopy commutative:
\begin{align}\label{dia3}
\xymatrix{
X_m \ar@{-->}[dd]_{\exists ~g} \ar[rr]^{\iota_m} && X_{m+1} \ar[dd]^{f} \\\\
X_m \ar[rr]^{\iota_m} && X_{m+1}
}
\end{align}
Thus, $g\in \AA^l(X_m) = \Aut(X_m)$. Therefore, $f\in \AA^m(X_{m+1}) = \Aut(X_{m+1})$. Hence, $$\NAA(X_{m+1})\leq l = \NAA(X_m).$$

\end{proof}

\begin{example}\label{exam}
Consider $X = S^n\vee S^{n+1}$ for $n\geq 2$. Observe that $X_n = S^n$ and $X_{n+1} = X$. From Proposition \ref{iihsn}, we have $\NAA(X) = \NAA(X_{n+1}) = n+1$ or $\NAA(X) = \NAA(X_{n+1})\leq \NAA(X_n) = n$. Let $f\colon X\to X$ be a map defined as 
\[f(x) = \begin{cases}
	x &\text{if} ~x\in S^n, \\
	\ast &\text{if}~~x\in S^{n+1}.
\end{cases} \]
Then clearly $f\in \AA^n(X)$ but $f\notin \AA^{n+1}(X)$. Hence $\NAA(X) = \NAA(X_{n+1}) = n+1$.

Moreover consider $Y = \displaystyle\bigvee^{\infty}_{i =2} S^k$. For any $m\geq 2$ we have $Y_m = \displaystyle\bigvee^m_{i =2} S^k$. Now we proceed by induction. From the above we have $\NAA(Y_3) = 3$ and let $\NAA(Y_l) = l$. Then $Y_{l+1} = Y_l \bigvee S^{l+1}$ and consider a map $g\colon Y_{l+1}\to Y_{l+1}$ defined as 
\[g(y) = \begin{cases}
	y &\text{if} ~ y\in Y_l, \\
	\ast & \text{if}~y\in S^{l+1}.
\end{cases} \]
Thus, clearly $g\in \AA^l(Y_{l+1})$ but $g\notin \AA^{l+1}(Y_{l+1})$. Hence $\NAA(Y_{l+1}) = l+1$ using Proposition \ref{iihsn}. Therefore by induction we have $\NAA(Y_m) = m$ for $m\geq 2$ and $\NAA(Y)=\infty$.

\end{example}

Recall that the homotopical dimension of $X$ is defined by $$\pi_{\ast} \text{-} \dim(X) := \max\big\{k:~ \pi_k(X)\neq 0\big\}.$$
	
Given any $f\colon X\to Y$, the induced maps of homology sections $f_n\colon X_n\to Y_n$ do not necessarily exist. In \cite{AIMHD}, the author gives some sufficient conditions to have induced maps in homology decomposition.
	
\begin{lemma}\label{egshe}
Let $X$ be a simply connected CW-complex such that $H_k(X)$ is finitely generated free abelian group for each $k$. If $\pi_{\ast} \text{-} \dim(X) = m$, then for each $n\geq m+1$ $$\Aut(X)\cong \Aut(X_n).$$
\end{lemma}	
	
\begin{proof}
Since $H_k(X)$ is free abelian group for all $k$, therefore homology decomposition of $X$ gives the CW-decomposition of $X$. Therefore for any map $f\colon X\to X$, we have $f'_n\colon X_n \to X_n$ such that $f\circ j_n \simeq j_n\circ f'_n$ by cellular approximations, where $f'_n\colon X_n\to X_n$ is the restriction of the cellular map $f'\colon X\to X$ on $X_n$ such that $f\simeq f'$. Note that if $f\colon X\to X$ is self-homotopy equivalence then $f'_n\colon X_n\to X_n$ is self-homotopy equivalence. Consider the map 
\begin{align}\label{mehs}
\phi \colon [X,X]\to [X_n,X_n]~~ \text{ defined as } \phi(f) = f'_n.
\end{align}
Let $n\geq m+1$. To show the surjectivity of $\phi$, let $r\in [X_n,X_n]$ then $r\colon X_n\to X_n\subset X.$ Since $\pi_n(X) = 0$ for all $n\geq m+1$, so $H^{k+1}\big(X,X_n,\pi_k(X)\big) = 0$ for all $k\geq n$. Therefore by Obstruction we have a map $\widetilde{r}\colon X\to X$ such that $\phi(\widetilde{r}) = r$. For injectivity, let $f,g\in [X,X]$ such that $\phi(f) = \phi(g)$. This implies $f'_n \simeq g'_n\colon X_n\to X_n\subset X$. Let $F\colon X_n\times I\to X$ be the homotopy map between $f'_n$ and $g'_n$.   Consider a map $G\colon X\times \partial{I}\cup X_n\times I\to X$ defined as 

\[ G(x,t) = \begin{cases}
	f'(x)  & \forall (x,t)\in X\times\{0\}, \\
	g'(x) & \forall (x,t)\in X\times\{1\},  \\
	F(x,t) & \forall (x,t)\in X_n\times I.
\end{cases} \]

\noindent Note that $H^{k+1}\big(X\times I, X\times \partial{I}\cup X_n\times I; \pi_k(X)\big) = 0$ for all $k\geq n+1$. Therefore there exists an extension map $\widetilde{G}\colon X\times I\to X$ of $G$ by obstruction theory. Hence $\widetilde{G}$ be the homotopy between $f'$ and $g'$. Thus $f\simeq g$, so $\phi$ is injective.
Consequently for $n\geq m+1$ we get the bijective map  $$\phi\colon [X,X]\to [X_n,X_n].$$
Moreover Observe that $\phi(f\circ g)=\phi (f'\circ g') = (f'\circ g')_n = f'_n\circ g'_n = \phi(f)\circ \phi(g).$
Hence $$\phi\colon \Aut(X)\xrightarrow{\cong} \Aut(X_n).$$	
\end{proof}

\begin{remark}
For a CW-decomposition of X, a self-homotopy equivalence f: X → X does not necessarily induce a self-homotopy equivalence on the skeleton \cite[Remark]{RSPHEC}. 
Further, if the CW-decomposition of $X$ comes from the homology decomposition (when all the homology groups of X are free abelian), then the induced map of f on each skeleton gives a self-homotopy equivalence.
\end{remark}

Recall the theorem of Serre (\cite{SGHCG}): For a simply connected CW-complex $X$, the homology group $H_k(X)$ is finitely generated abelian group for all $k$ if and only if $\pi_k(X)$ is finitely generated abelian group for all $k$.

\begin{lemma}\label{ihhsnhs}
Let $X$ be a simply connected $CW$ complex such that $H_k(X)$ is finitely generated for each $k\in \mathbb{N}$. Then $\NA(X_n) = \NAA(X_n)\leq n$.	
\end{lemma}

\begin{proof}
Let $H_k(X)$ is finitely generated abelian group for each $k$. Therefore from the definition of homology decomposition we have $H_k(X_n)$ is finitely generated abelian group for each $k$. Hence we get the desired result using \cite[Corollary 42]{OYSCFC}.
	
\end{proof}

The following Theorem deduce the relation between homology self-closeness number of $X$ with its homology section. 
\begin{theorem}\label{echd}
Let $X$ be a simply connected CW-complex such that $\pi_{\ast} \text{-} \dim(X) = m$.	Then following holds:
\begin{enumerate}[{(a)}]\setlength\itemsep{0.5 em}
\item $2\leq \NAA(X)\leq m+1$.
\item $2\leq\NAA(X)\leq m$, if the group $\pi_m(X)$ is finitely generated.
\item $\NAA(X_n) = \NAA(X)$, if $n\geq m+1$ and $H_{k}(X)$ is finitely generated free abelian group for each $k$.
\item $\NAA(X_n)\leq \NAA(X)$, if $n\geq m+1$ and $H_k(X)$ is finite generated for each $k$.
\item $\NAA(X_n)\leq m$ for $n\leq m$.
\item $\NAA(X)\leq \NAA(X_n)$, if $\dim(X)\leq n < \infty$.
\end{enumerate}
\end{theorem}

\begin{proof}
\begin{enumerate}[{(a)}]\setlength\itemsep{1 em}
\item Since  $X$ is a simply connected space, thus $\NAA(X)\geq 2$. Let $f\in \AA^{m+1}(X)$ then we have $f\in \A^m(X) = \Aut(X)$. Hence $2\leq \NAA(X)\leq m+1$. 
\item Let $f\in \AA^m(X)$. Then $f_{\ast}\colon H_k(X)\to H_k(X)$ is isomorphism for all $k\leq m$. Thus, $f_{\#}\colon \pi_k(X)\to \pi_k(X)$ is isomorphism for all $k\leq m-1$ and $f_{\#}\colon \pi_m(X)\to \pi_m(X)$ is surjective. Since $\pi_m(X)$ is finitely generated therefore $f_{\#}\colon \pi_m(X)\to \pi_m(X)$ is also an isomorphism. So, $f\in \A^m(X) = \Aut(X)$ and we get the desired result.	
\item Note that $\pi_k(X)$ is finitely generated abelian group for all $k$ as $H_{k}(X)$ is finitely generated abelian group for each $k$ by Serre. Assume that $n\geq m+1$. Let $\NAA(X) = l\leq m$ and $g\in \AA^l(X_n)$. Then $\phi^{-1}(g)\in \AA^l(X) = \Aut(X)$ from Lemma \ref{egshe}. Thus $g\in \Aut(X_n)$. So $\NAA(X_n)\leq l=\NAA(X)$.  If possible, assume that $\NAA(X_n)<l$. Let $h\in \AA^{l-1}(X)$ then $\phi(h)\in \AA^{l-1}(X_n) = \Aut(X_n)$. Therefore, $h\in \Aut(X)$, which contradict the fact that $\NAA(X) = l$. Therefore, for $n\geq m+1$, we have $$\NAA(X_n) = \NAA(X).$$
\item Let $\NAA(X) = r$ and $h\in \A^r(X_n)$. Since $n\geq m+1$, as in the proof of Lemma \ref{egshe}, we see that there exists a map $\bar{h}\colon X\to X$ such that the following diagram is homotopy commutative.
\begin{align}\label{dia9}
	\xymatrix{
		X \ar[rr]^{\bar{h}} && X  \\\\
		X_n \ar[uu]^{j_n} \ar[rr]_{h} && \ar[uu]_{j_n} X_n
	}
\end{align}
Note that the map $j_{n\ast}\colon H_i(X_n)\xrightarrow{\cong} H_i(X)$ for all $i\leq n$ and $r\leq m < n$. Thus, $\bar{h}\in \A^r(X) = \Aut(X).$ Therefore using the commutativity of the diagram we have $h\in \A^n(X_n) = \Aut(X_n)$. Hence $\NAA(X_n)\leq r = \NAA(X).$
\item Observe that $\NAA(X_n)\leq n$. Therefore we get the desired result.
\item Let $\dim(X) < \infty$ and $n\geq \dim(X)$. Thus $H^{n+1}\big(X_n, H_{n+1}(X)\big) = 0$. Therefore, for any map $f\colon X\to X$ there exists a map $f_n\colon X_n\to X_n$ such that we have a diagram as of Equation \ref{dia9}, which is homotopy commutative \cite[Theorem 3.1]{AIMHD}.
Let $\NAA(X_n) = k$ and $h\in \A^k(X)$. Using the commutativity of the diagram, we have $h_n\in \A^k(X_n) =\Aut(X_n)$. Further, using the diagram we get $h\in \A^n(X) = \Aut(X)$ since $\NAA(X) \leq \dim(X).$ Hence, $\NAA(X) \leq k = \NAA(X_n).$ 
\end{enumerate}
\end{proof}	
	
\begin{example}
Let $X = \mathbb{C}P^{\infty}.$ Then $\pi_{\ast} \text{-} \dim(X) = 2$. Note that 
\[ H_k(\mathbb{C}P^{\infty}) = 
\begin{cases}
	\mathbb{Z} &\text{if} ~k  \text{ is even,} \\
	 0 &\text{if}~k \text{ is odd}.
\end{cases} \] 
From Theorem \ref{echd}, we have $\NAA(X) = \NAA(X_n)$ for all $n\geq 3$, where $X_n$ is the $n$-th homology decomposition of $X$. Observe that $X_{2n+1} = X_{2n} = \mathbb{C}P^n$ for all $n\geq 1$. Therefore, $$\NAA(\mathbb{C}P^{\infty}) = \NAA(X) = \NAA(X_3) = \NAA(X_2) = \NAA(\mathbb{C}P^1) = \NAA(S^2) = 2.$$  
Moreover, $\NAA(\mathbb{C}P^n) = \NAA(X) = 2$ for all $n\geq 2$ by Theorem \ref{echd}.
Consequently,
$$\NAA(\mathbb{C}P^n) = \NAA(X_{2n}) = 2, ~~\text{for}~ n\in \mathbb{N}\cup \{\infty\}.$$
\end{example}	

\begin{remark}\label{rmk8}
Consider $X = \mathbb{C}P^{\infty}$. Then $X_{2n+1} = X_{2n} = \mathbb{C}P^n$ for $n\geq 1$. Therefore, 
\[ \NAA(X_n) = 
\begin{cases}
	0 &\text{if} ~~n<2, \\
	2 &\text{if}~~ n\geq 2.
\end{cases} \] 
\end{remark}

\begin{example}
Consider $X = K(G,n)$ for $n\geq 2$. Note that the Hurewicz map $h_{\ast}\colon \pi_{n+1}(X)\to H_{n+1}(X)$ is epimorphism, thus $H_{n+1}(X) = 0$. For homology sections $\{X_m\}$, we have
\[ X_m = 
\begin{cases}
	\ast &\text{if} ~~m<n, \\
	 M(G,n) &\text{if}~~ m = n, n+1.
\end{cases} \] 
By definition, the map $j_{n+2}\colon X_{n+2}\to X$ is a homologial $(n+2)$-equivalence, and hence a homotopical $(n+2)$-equivalence (cf. Lemma \ref{leq}). So $j_{n+2\#}\colon\pi_k(X_{n+2})\xrightarrow{\cong} \pi_k(X)$ is an isomorphism for all $k\leq n+1$. Hence, $X_{n+2}$ is $(n-1)$-connected CW-complex. Therefore,
$$\A^{n-1}(X_{n+2}) = \cdots = \A^1(X_{n+2}) = [X_{n+2}, X_{n+2}].$$ 
Thus, $\NA(X_{n+2})\geq n$. Using Lemma \ref{ihhsnhs}, we get $\NAA(X_{n+2}) = \NA(X_{n+2}) \geq n$. 

Moreover, from Theorem \ref{echd}(d), we have $\NAA(X_{n+2})\leq \NAA(X) = \NA(X) = n$. Hence, $\NAA(X_{n+2}) = n$. Similarly, we can prove that $\NAA(X_m) = n$ for $m > n+2$. Consequently,
\[ \NAA(X_m) = 
\begin{cases}
	0 &\text{if} ~~m<n, \\
	n &\text{if} ~~ m\geq n.
\end{cases} \] 
\end{example}

\sect{Homotopy decomposition}\label{sec4}

In this section, we consider homotopy decomposition (Postnikov tower) of a space $X$ which was studied in previous (cf. \cite{MBGSE}). This can be thought as a dual construction to homology decomposition. Here, we prove some results related to homotopy and homology decomposition and their intermediate relations. These results help us to compute homology self-closeness numbers of Postnikov tower, some of them can be thought as generalization of those results in \cite{CLSCWHD}. First, recall the definition

\begin{definition}
Let $X$ be a simply connected CW-complex. A homotopy decomposition (Postnikov decomposition) of $X$ consists of a sequence of simply connected CW-complexes $\{X^{(n)}\}$ and maps $g_n\colon X\to X^{(n)}$ such that 

\begin{enumerate}[{(i)}]\setlength\itemsep{1 em}
	\item Each inclusion map $g_n\colon X\to X^{(n)}$ induces isomorphisms  $g_{n\#}\colon \pi_k(X)\to \pi_k(X^{(n)})$ for all $k\leq n$ and $\pi_k(X^{(n)}) = 0$ for $k>n$.
	\item There exist maps $p_{n+1}\colon X^{(n+1)}\to X^{(n)}$ such that $$K\big(\pi_{n+1}(X),n+1\big)\to X^{(n+1)}\xrightarrow{p_{n+1}}X^{(n)}$$ is a fiber sequence.
	\item The following diagram is commutative
	$$
	\xymatrix{
		&&X^{(n+1)} \ar[dd]^{p_{n+1}} \\\\
		X \ar[uurr]^{g_{n+1}} \ar[rr]_{g_n} &&   X^{(n)}   
	}
	$$
\item The fibration sequence in (ii) is equivalent to the principal fibration 
$$K\big(\pi_{n+1}(X),n+1\big)\to X^{(n+1)}\to X^{(n)}\xrightarrow{k^{n+1}} K\big(\pi_{n+1}(X),n+2\big),$$
 determined by a map $k^{n+1}$, where $[k^{n+1}]\in H^{n+2}\big(X^{(n)};\pi_{n+1}(X)\big)$.
\end{enumerate}
\end{definition}
\noindent We write this as $\{X^{(n)},g_n,p_n,k^n\}$ and call it a homotopy decomposition (Postnikov decomposition) of $X$. The maps $k^{n+1}\colon X^{(n)}\to K\big(\pi_{n+1}(X),n+2\big)$ are called $k$-invariants. The space $X^{(n)}$  is called the $n$-th homotopy section and is obtained as the homotopy fiber of $k^{n}.$

\begin{remark}\label{rem10}
From the definition, we have $ \NA(X^{(n)})\leq n.$ 
\end{remark}

The following Lemma establish a relation between the homology self-closeness numbers of homotopy and homology sections.

\begin{lemma}\label{ehhsn}
Let $X$ be a simply connected CW-complex such that $H_k(X)$ is finitely generated abelian group for each $k$. Then the following results hold.
\begin{enumerate}[(i)]\setlength\itemsep{1 em}
	\item $\NAA(X^{(n)}) = \NA(X^{(n)})\leq n.$
	\item $\NAA(X^{(n)})\leq \NAA(X_n),$ if $H_n(X)$ is free-abelian group.
\end{enumerate}
\end{lemma}

\begin{proof}
\begin{enumerate}[(i)]\setlength\itemsep{1 em}
\item From the definition of homotopy decomposition, we have $\pi_k(X^{(n)})$ is finitely generated for all $k$. Therefore using \cite[Corollary 42]{OYSCFC}, we get the desired result.
	
\item Let $j_n\colon X_n\to X$ and $g_n\colon X\to X^{(n)}$ be two maps over the homology decomposition and homotopy decomposition of $X$ respectively. Observe that $g_n\colon X\to X^{(n)}$ is homotopical $(n+1)$-equivalence.  From Lemma \ref{leq}, we say that $g_n$ is homological $(n+1)$-equivalence. Therefore, $(g_n\circ j_n)_{\ast}\colon H_i(X_n)\xrightarrow{\cong} H_i(X^{(n)})$ for all $i\leq n$. Moreover, for any $h\colon X^{(n)}\to X^{(n)}$ there exists a $\bar{h}\colon X_n\to X_n$ such that the following diagram is homotopy commutative by Compression Lemma.
\begin{align}
	\xymatrix{
			X^{(n)} \ar[rr]^{h} && X^{(n)}  \\\\
			X_n \ar[uu]^{g_n\circ j_n} \ar[rr]_{\bar{h}} && \ar[uu]_{g_n\circ j_n} X_n
		}
\end{align}
Let $\NAA(X_n) = k$ and $f\in \AA^k(X^{(n)})$. Thus, $\bar{f}\in \AA^k(X_n) = \Aut(X_n)$. Therefore, $f\in \AA^n(X^{(n)}) = \Aut(X^{(n)}).$ Hence, $$\NAA(X^{(n)})\leq \NAA(X_n).$$
\end{enumerate}
\end{proof}

Next we deduce relations between the homology self-closeness numbers of two consecutive homotopy sections in Postnikov tower. 

\begin{theorem}\label{hsnchs}
Let $X$ be a simply connected CW-complex such that $H_k(X)$ is finitely generated for all $k$. For homotopy decomposition $\{X^{(n)},g_n,p_n,k^n\}$ of $X$, we have the following:

\begin{enumerate}[{(a)}]\setlength\itemsep{1 em}
\item If the induced map $$k^{m+1\#}\colon \big[K\big(\pi_{m+1}(X), m+2\big), K\big(\pi_{m+1}(X), m+2\big)\big]\to \big[X^{(m)},K\big(\pi_{m+1}(X), m+2\big)\big]$$ is surjective  for some $m\in \mathbb{N}$, then $\NAA(X^{(m)})\leq \NAA(X^{(m+1)})$.

\item If the induced map $$k^{m+1\#}\colon \big[K\big(\pi_{m+1}(X), m+2\big), K\big(\pi_{m+1}(X), m+2\big)\big]\to \big[X^{(m)},K\big(\pi_{m+1}(X), m+2\big)\big]$$ is injective such that $$k^{m+1}_{\#}(\Aut(X^{(m)})) \subseteq k^{m+1\#}\big(\big[K\big(\pi_{m+1}(X), m+2\big), K\big(\pi_{m+1}(X), m+2\big)\big]\big),$$ then $\NAA(X^{(m+1)})\leq \NAA(X^{(m)})$.

\item If the induced map $k^{m+1\#}$ is bijective, then $\NAA(X^{(m)}) = \NAA(X^{(m+1)})$.
\end{enumerate}
\end{theorem}

\begin{proof}
Note that $\NAA(X^n) = \NA(X^n)$ for homotopy sections $\{X^n\}$ by Lemma \ref{ehhsn}. Rest of the proof directly follows from \cite[Theorem 3.3 \& Theorem 3.9]{OYSF}.
	
\end{proof}


\begin{proposition}\label{ihsncs}
Let $X$ be a simply connected CW-complex such that $H_k(X)$ is finitely generated for all $k$. Then either $\NAA(X^{(n+1)}) = n+1$ or $\NAA(X^{(n+1)})\leq \NAA(X^{(n)})$.
\end{proposition}

\begin{proof}
Observe that the map $p^{\#}_{n+1}\colon [X^{(n)}, X^{(n)}]\to [X^{(n+1)}, X^{(n)}]$ is surjective, see \cite[Proposition 8.2.2]{AIH}. As in the proof of Proposition \ref{iihsn}, we have either $\NA(X^{(n+1)}) = n+1$ or $\NA(X^{(n+1)})\leq \NA(X^{(n)})$. Moreover, using Lemma \ref{ehhsn}, we get the desired result.

\end{proof}

The following Theorem is a homological version of \cite[Theorem 3.5]{CLSCWHD}.
\begin{theorem}\label{hshd}
Let $X$ be a simply connected CW-complex  of dimension $m$ such that $H_k(X)$ is finitely generated for each $k$. Then
\begin{enumerate}[{(a)}]\setlength\itemsep{1 em}
	\item $\NAA(X^{(n)}) = \NAA(X)$, if the conditions hold $\NAA(X)\leq n$ and $m-1\leq n$.
	\item $\NAA(X)\leq \NAA(X^{(n)})$, if the condition holds $\NAA(X)\leq n < m - 1$.
	\item $\NAA(X^{(n)})< \NAA(X)$, if the condition holds $n<\NAA(X)$.
\end{enumerate}
\end{theorem}

\begin{proof}
\begin{enumerate}[(a)]\setlength\itemsep{1 em}
\item We want to prove the inequality first $\NAA(X^{(n)})\leq \NAA(X)$. Assume that $\NAA(X) = k$. Consider the case $k < n$, otherwise we get the inequality trivially. Given any map $\phi\colon X^{(n)}\to X^{(n)}$ there exists a map $\bar{\phi}\colon X\to X$ such that the following diagram is homotopy commutative by Compression Lemma \cite[Lemma 4.6]{AHAT}.
\begin{align}\label{dia8}
	\xymatrix{
		X^{(n)} \ar[rr]^{\phi} && X^{(n)}  \\\\
		X \ar[uu]^{g_n} \ar[rr]_{\bar{\phi}} && \ar[uu]_{g_n} X
		       }
\end{align}
Let $f\in \AA^k(X^{(n)})$. Then there exists a $\bar{f}\in \AA^k(X)$ by Compression Lemma and the fact that $g_{n\ast}\colon H_i(X)\xrightarrow{\cong} H_i(X^{(n)})$ for all $i\leq n$. 
Since $\bar{f}\in \AA^k(X) = \Aut(X)$, therefore using the commutativity of the diagram, we have $f\in \AA^n(X^{(n)}) = \Aut(X^{(n)}).$ Hence $\NAA(X^{(n)})\leq \NAA(X).$ 
	
Conversely, let $\NAA(X^{(n)}) = l$ and $h\in \AA^l(X)$. Then, there exists $h^n\colon X^{(n)}\to X^{(n)}$ which makes the above Diagram \ref{dia8} homotopy commutative (cf. \cite{KIPS}). Thus, $h^n\in \AA^l(X^{(n)}) = \Aut(X^{(n)}).$ Therefore, $h\in \AA^n(X) = \Aut(X)$. Hence $\NAA(X)\leq \NAA(X^{(n)})$ and we get the desired result.
\item Follows from the converse part of (a).
\item Observe that $\NAA(X^{(n)})\leq n$. Thus, we get the desired result.
\end{enumerate}
\end{proof}

\begin{corollary}\label{ehshh}
Let $X$ be a simply connected CW-complex such that $H_k(X)$ is finitely generated for all $k$. Assume that any one of the following conditions hold:
\begin{enumerate}[{(a)}]\setlength\itemsep{1 em}
	\item $\pi_{\ast} \text{-} \dim(X) <n < \infty$ and $H_k(X)$ is free for each $k$.
	\item $\dim(X) \leq n < \infty$.
\end{enumerate}	
Then, $\NAA(X^{(n)}) = \NAA(X_n)$.
\end{corollary}

\begin{proof}
\begin{enumerate}[{(a)}]\setlength\itemsep{1 em}
 \item Observe that the structure map of the Postnikov tower $g_n\colon X\to X^{(n)}$ is a homotopy equivalence if $n \geq \pi_{\ast} \text{-} \dim(X)$. Thus for $n\geq \pi_{\ast} \text{-} \dim(X)$, we have $\NAA(X) = \NAA(X^{(n)})$. Therefore from Theorem \ref{echd} we have $\NAA(X^{(n)}) = \NAA(X_n)$ for $n\geq \pi_{\ast} \text{-} \dim(X) +1$.
 \item Note that for $n\geq \dim(X)$, we have $X_n = X$. Hence using Theorem \ref{hshd}, we get the desired result.
\end{enumerate}
	
\end{proof}


%
%


\begin{example}
For $n\in \mathbb{N}\cup \{\infty\}$, consider $X = \mathbb{C}P^n$. From Lemma \ref{ehhsn}(ii), we have $$\NAA(X^{(k)})\leq \NAA(X_k) ~\text{ for all } k.$$  For $k\geq 2$, we know that $X^{(k)}$ are simply connected, thus $\NAA(X^{(k)})\geq 2$. Therefore,  $$2\leq \NAA(X^{(k)})\leq \NAA(X_k) = 2 ~\text{ for } k\geq 2.$$ Hence,
\[ \NAA(X^{(k)}) = 
\begin{cases}
	0 &\text{if}~ k<2, \\
	2 &\text{if}~ k\geq 2.
\end{cases} \]
\end{example}

\begin{example}
Let $G$ and $H$ be finitely generated abelian groups. Consider $X = M(G,n)\bigvee M(H,m)$, where $m>n>1$. Observe that
\[ X^{(k)} = 
\begin{cases}
\ast  &\text{if}~ k<n, \\
K(G,n) &\text{if}~ k = n.
\end{cases} \]
Moreover, $H_k(X) = 0$ for all $n<k<m$. Therefore,
\[ X_k = 
\begin{cases}
  \ast &\text{if}~ k<n, \\
  M(G,n) &\text{if}~ n\leq k <m, \\
  X &\text{if} ~ k = m.
\end{cases} \]
Note that, for $k\geq n$, we have $n\leq \NAA(X^{(k)})$  since $X^{(k)}$ is $n-1$ connected. From Lemma \ref{ehhsn}(ii), we have $\NAA(X^{(k)})\leq \NAA(X_k)$ for $n < k <m$. Thus, 
$$n\leq \NAA(X^{(k)})\leq \NAA(X_k) = \NAA(M(G,n)) = n, ~~\text{for}~ n < k <m.$$ 

Note that $\AA^n(X) = \cdots = \AA^{m-1}(X)$ and $\Aut(X) = \AA^m(X) \subsetneq \AA^{m-1}(X)$, see Example \ref{exam}. Therefore, $\NAA(X) = m$

From Theorem \ref{hshd}(a), we have $\NAA(X^{(k)}) = \NAA(X)$ for all $k\geq m$. Hence, $\NAA(X^{(k)}) = \NAA(X) = m$ for $k\geq m$.
Consequently,
\[ \NAA(X^{(k)}) = 
\begin{cases}
	0 &\text{if}~ k<n, \\
	n &\text{if}~ n\leq k <m, \\
	m &\text{if}~ k\geq m.
	
\end{cases} \]

\end{example}

In general we have the following result.
\begin{proposition}\label{hsnmm}
Let $G_1,\cdots, G_m$ be finitely generated abelian groups. Consider $X= \displaystyle\bigvee^m_{i=1} M(G_i, n_i)$, where $n_{i+1}>n_i>1$. Then 
\[ \NAA(X^{(k)}) = 
\begin{cases}
	0 &\text{if}~ k<n_1, \\
	n_i &\text{if}~ n_i\leq k <n_{i+1}, \\
	n_m &\text{if} ~ k\geq n_m.
\end{cases} \]
\end{proposition}

\begin{proof}
Note that $\AA^{n_{i+1}}(X)\subsetneq \AA^{n_i}(X)$ for $i = 1,\ldots, m-1$, see Example \ref{exam}. We have already proved for the case $m=2$. It is sufficient to prove the result for the case $m = 3$ and then, we can say that the result holds inductively. Let $X= \bigvee\limits_{i=1}^3 M(G_i, n_i).$ Observe that 
\[ X^{(k)} = 
\begin{cases}
	\ast &\text{ if } k<n_1,\\
	K(G_1,n_1) &\text{ if } k = n_1.
\end{cases} \]

Further, 
\[ X_k = 
\begin{cases}
	\ast &\text{ if } k < n_1,\\
	\bigvee\limits_{j=1}^i M(G_j, n_j) &\text{ if } n_i\leq k < n_{i+1},\\
	X &\text{ if } k = n_m.
\end{cases}
\]

Moreover, $X^{(k)}$ is $(n_1 - 1)$ connected for $k\geq n_1$. Thus, $n_1\leq \NAA(X^{(k)})$ for $k\geq n_1$. From Lemma \ref{ehhsn}(ii), we have $$n_1\leq \NAA(X^{(k)})\leq \NAA(X_k) = \NAA(X_{n_1}) = n_1, ~\text{for}~ n_1<k<n_2.$$

\noindent Using Proposition \ref{ihsncs}, we have $$\NAA(X^{(n_2)}) = n_2~ \text{or} ~\NAA(X^{(n_2)})\leq \NAA(X^{(n_2 - 1)}) = n_1.$$  
If possible, assume that $\AA^{n_1}(X^{(n_2)}) = \Aut(X^{(n_2)})$. Let $f\in \AA^{n_1}(X)$ but $f\notin \AA^{n_2}(X)$. By a functorial construction of Postnikov tower, there exist $f^{(n_2)}$ which makes the following diagram homotopy commutative ( cf. \cite[Proposition 7.2.11]{AIH} ).
\begin{align}\label{dia5}
\xymatrix{
	X \ar[dd]_{g_{n_2}} \ar[rr]^{f} && X \ar[dd]^{g_{n_2}} \\\\
	X^{(n_2)} \ar@{-->}[rr]_{\exists~ f^{(n_2)}} && X^{(n_2)}
}
\end{align}
Thus, $f^{(n_2)}\in \AA^{n_1}(X^{(n_2)}) = \Aut(X^{(n_2)})$. This contradicts the fact that $f\in \AA^{n_2}(X)$. Hence $\NAA(X^{(n_2)}) = n_2$. Further, using Lemma \ref{ehhsn}(ii), we have $$\NAA(X^{(k)})\leq \NAA(X_k) = \NAA(X_{n_2}) = n_2, ~\text{for}~ n_2< k <n_3.$$ 
If possible, assume that $\NAA(X^{(k)}) < n_2$ for some $k\in (n_2, n_3)$. Then $$\AA^{n_1}(X^{(k)}) = \cdots = \AA^{n_2 - 1}(X^{(k)}) = \Aut(X^{(k)}) ~(\text{since}~ H_j(X^{(k)}) \cong H_j(X)~ \forall j\leq k).$$ 
Let $h\in \AA^{n_1}(X)$ but $h\notin \AA^{n_2}(X)$. Then by the similar argument as above, we arrive at contradiction. Hence $\NAA(X^{(k)}) = n_2$ for $n_2< k <n_3$.

From Theorem \ref{hshd}(a), we have $\NAA(X^{(k)}) = \NAA(X) = n_3$ for all $k\geq n_3$. Consequently 
\[ \NAA(X^{(k)}) = 
\begin{cases}
	0 &\text{if}~ k<n_1, \\
	n_1 &\text{if} ~ n_1\leq k <n_2, \\
	n_2 &\text{if}~ n_2\leq k <n_3, \\
	n_3 &\text{if} ~ k\geq n_3.
\end{cases} \]

\end{proof}

\begin{example}
Consider $X = S^m\times S^n$, where $m>n>1$. Observe that 
\[ X^{(k)} =
\begin{cases}
 \ast &\text{if}~ k<n, \\
 K(\mathbb{Z},n) &\text{if}~ k= n.
\end{cases} \]
Therefore $\NAA(X^{(k)}) = 0$ for all $k<n$. Note that $X^{(k)}$ is $n-1$ connected for $k\geq n$. So $n\leq \NAA(X^{(k)})$ for all $k\geq n$. From  Lemma \ref{ehhsn}(ii), we have 
$$n\leq \NAA(X^{(k)})\leq \NAA(X_k) = \NAA(X_n) = \NAA(S^n) = n ~\text{for all}~n\leq k<m.$$

Observe that $\NAA(X) = \NAA(S^m\times S^n) = \NA(S^m\times S^n) = m$ ~(cf. \cite[Proposition 5]{OYSEC}).

Moreover, using Theorem \ref{hshd}(a), we have $\NAA(X^{(k)}) = \NAA(X)$ for all $k\geq m+n-1$. Therefore,
$$\NAA(X^{(k)}) = m ~\text{for all}~ k\geq m+n-1.$$

Further, for $m\leq k<m+n-1$, we have $m\leq \NAA(X^{(k)})$ from Theorem \ref{hshd}(b) and using Lemma \ref{ehhsn}(ii), we have $$\NAA(X^{(k)})\leq \NAA(X_k) = \NAA(X_m) = \NAA(S^m \vee S^n) = m.$$
Hence, $\NAA(X^{(k)}) = m$ for all $m\leq k<m+n-1$. Consequently, 
\[ \NAA(X^{(k)}) = 
\begin{cases}
	0 &\text{if}~ k<n, \\
	n &\text{if}~n\leq k <m, \\
	m &\text{if}~ k\geq m.
\end{cases} \]
\end{example}

\begin{example}
For $l > 0$, consider $X = \Sigma^l (S^m\times S^n)$, where $m > n >1$. Note that $$\Sigma (S^m\times S^n) \simeq S^{n+1} \bigvee S^{m+1} \bigvee S^{m+n+1} ~~\text{(cf. \cite[Proposition 4I.1]{AHAT})}.$$ 
Therefore, $X = S^{n+l} \bigvee S^{m+l} \bigvee S^{m+n+l}$. From Proposition \ref{hsnmm}, we have 
\[ \NAA(X^{(k)}) = 
\begin{cases}
	0 &\text{if} ~ k< n+l, \\
	n+l &\text{if} ~n+l\leq k <m+l, \\
	m+l &\text{if} ~m+l \leq k <m+n+l, \\
	m+n+l &\text{if} ~ k\geq m+n+l.
\end{cases} \]
\end{example}

\end{document}